\documentclass[11pt]{amsart}
\usepackage[pdftex]{hyperref}
\usepackage{amsfonts, amsthm, amsmath, amssymb}
\usepackage[all]{xy}
\usepackage[top=1in,bottom=1in,left=1in,right=1in]{geometry}

\newcommand{\bB}{\mathbb{B}}
\newcommand{\bC}{\mathbb{C}}

\newcommand{\bF}{\mathbb{F}}

\newcommand{\bM}{\mathbb{M}}

\newcommand{\bP}{\mathbb{P}}
\newcommand{\bQ}{\mathbb{Q}}
\newcommand{\bR}{\mathbb{R}}

\newcommand{\bZ}{\mathbb{Z}}

\newcommand{\fH}{\mathfrak{H}}

\newcommand{\fm}{\mathfrak{m}}

\newcommand{\Aut}{\operatorname{Aut}}

\newcommand{\Ell}{\operatorname{Ell}}

\newcommand{\Tr}{\operatorname{Tr}}

\newtheorem{thm}{Theorem}[section]
\newtheorem{conj}[thm]{Conjecture}

\theoremstyle{definition}

\theoremstyle{remark}

\begin{document}

\title{Monstrous Moonshine over $\mathbb{Z}$?}
\author{Scott Carnahan}

\begin{abstract}
Monstrous Moonshine was extended in two complementary directions during the 1980s and 1990s, giving rise to Norton's Generalized Moonshine conjecture and Ryba's Modular Moonshine conjecture.  Both conjectures have been unconditionally established in the last few years, so we describe some speculative conjectures that may extend and unify them.
\end{abstract}

\maketitle

\section{Monstrous Moonshine and its extensions}

Monstrous moonshine began in the 1970s with an examination of the modular $J$-invariant, which we normalize as:
\[ J(\tau) = q^{-1} + 196884q + 21493760q^2 + \cdots \]
Here, $\tau$ ranges over the complex upper half-plane $\fH$, and the expansion in powers of $q = e^{2\pi i \tau}$ is the Fourier expansion.  The $J$-function is a Hauptmodul for $SL_2(\bZ)$, meaning $J$ is invariant under the standard action of $SL_2(\bZ)$ on $\fH$ by M\"obius transformations, and induces a complex-analytic isomorphism from the quotient $SL_2(\bZ) \backslash \fH$ to the complex projective line $\bP^1_{\bC}$ minus finitely many points.

McKay first observed a possible relationship between $J$ and the Monster sporadic simple group $\bM$, pointing out that the irreducible representations of $\bM$ have lowest dimension $1$ and $196883$, and that $196884 = 196883 + 1$.  Thompson then showed that the next few coefficients of $J$ also decompose into a small number of irreducible representations of $\bM$ \cite{T79}.  This suggested that the coefficients of $J$, which are all non-negative integers, describe a special class of representations of $\bM$.

The observations by McKay and Thompson were explained by the construction of the Monster vertex operator algebra $V^\natural = \bigoplus_{n=0}^\infty V^\natural_n$ in \cite{FLM88}.  The appearance of $J$ comes from the fact that $\sum_{n=0}^\infty \dim V^\natural_n q^{n-1} = J(\tau)$, and the appearance of the monster comes from the fact that $\Aut(V^\natural) \cong \bM$.

Borcherds showed in \cite{B92} that $V^\natural$ encoded a stronger relationship between modular functions and the monster, showing that for any $g \in \bM$, the McKay-Thompson series
\[ T_g(\tau) := \sum_{n=0}^\infty \Tr(g|V^\natural_n) q^{n-1} \]
is a Hauptmodul for some discrete group $\Gamma_g < SL_2(\bR)$.

This resolved the main ``Monstrous Moonshine'' conjecture posed in \cite{CN79}, and it also partially resolved the ``Jack Daniels'' question posed in \cite{O74}.  As Ogg noted, the 15 primes dividing the order of $\bM$ are precisely the primes $p$ such that the modular curve $X_0(p)^+$ is genus zero, and he offered a bottle of Jack Daniels for an explanation.  The explanation for the monstrous primes being contained among the genus zero primes is: for each such prime $p$, there is a conjugacy class $p$A in the monster, and Borcherds's theorem established that for any $g \in p$A, the McKay-Thompson series $T_g(\tau)$ is a Hauptmodul for $X_0(p)^+$.  The reverse containment is unlikely to be resolved without a good understanding for why $\bM$ is as big as it is.

One Moonshine-like phenomenon that Borcherds's theorem did not explain was that for $g \in p$A, the $q$-expansion of $T_g(\tau)$ has non-negative integer coefficients, and furthermore, these coefficients appear to decompose well into dimensions of representations of the centralizer $C_\bM(g)$.  For example, when $p=2$, the Hauptmodul for $X_0(2)^+$ has the $q$-expansion
\[ q^{-1} + 4372q + 96256q^2 + \cdots \]
while the centralizer $2.\bB$, a central extension of the Baby monster sporadic group, has irreducible representations of dimension $1, 4371, 96255, 96256, \ldots$.  This suggests that the Hauptmodul for $g$ in the class $p$A is given by the graded dimension of some natural representation of $C_\bM(g)$.  In fact, there were two conjectures, both resolved recently, either of which could be called an explanation of this phenomenon:

\begin{itemize}
\item Generalized Moonshine (conjectured in \cite{N87}): There is a rule that attaches to each $g \in \bM$, a graded projective representation $V(g) = \bigoplus_{n \in \bQ} V(g)_n$ of $C_\bM(g)$, and to each commuting pair $(g,h)$ in $\bM$ a holomorphic function $Z(g,h;\tau)$ on $\fH$ such that:
\begin{enumerate}
\item The series $\sum_n \Tr(\tilde{h}|V(g)_n) q^{n-1}$ converges normally on $\fH$ to $Z(g,h;\tau)$, for $\tilde{h}$ some lift of $h$ to a linear transformation of $V(g)$.
\item The assignment $(g,h) \mapsto Z(g,h;\tau)$ is invariant under simultaneous conjugation up to rescaling.
\item Each function $Z(g,h;\tau)$ is either constant or a Hauptmodul.
\item For any $\left(\begin{smallmatrix} a & b \\ c & d \end{smallmatrix} \right) \in SL_2(\bZ)$, $Z(g,h;\frac{a\tau+b}{c\tau+d})$ is proportional to $Z(g^a h^c, g^b h^d; \tau)$.
\item $Z(g,h;\tau)$ is proportional to $J(\tau)$ if and only if $g=h=1$.
\end{enumerate} 
\item Modular Moonshine (conjectured in \cite{R96}): For $g \in p$A, there is a $\bZ_{\geq 0}$-graded vertex algebra ${}^gV = \bigoplus_{n=0}^\infty {}^gV_n$ over $\bF_p$ such that for any $p$-regular element $h \in C_\bM(g)$, the graded Brauer character $\widetilde{\Tr}(h|{}^gV_n)q^{n-1}$ is equal to the McKay-Thompson series $T_{gh}(\tau)$.
\end{itemize}

These conjectures were solved along somewhat complementary paths.  For Generalized Moonshine, the candidate representation $V(g)$ was more or less identified as the irreducible $g$-twisted $V^\natural$-module $V^\natural(g)$ in \cite{DGH88}, and these twisted modules were proved to exist for all $g \in \bM$ in \cite{DLM97}, resolving claims 1, 2, and 5.  However, the main conjectured properties were only resolved recently, with claim 4 proved in \cite{CM16} and claim 3 in \cite{GM4}.  For Modular Moonshine, the vertex algebras were identified as Tate cohomology groups $\hat{H}^0(g,V^\natural_{\bZ})$, and the conjectured properties were proved in \cite{BR96}, \cite{B98}, and \cite{B99} under the assumption of existence of a self-dual integral form $V^\natural_{\bZ}$ of $V^\natural$.  The existence of the essential object was only resolved recently, in \cite{C17}.

In both cases, substantial progress was made in the late 1990s, but complete solutions were not possible due to the underdeveloped state of the theory of vertex operator algebras at the time.  The most necessary piece of technology was a solid theory of cyclic orbifolds, worked out in \cite{vEMS}.  This in turn depended on many foundational developments in the intervening years, such as the Verlinde formula \cite{H04}, a theory of logarithmic tensor products (e.g., \cite{HLZ6}), $C_2$-cofiniteness for fixed-point subVOAs \cite{M13}, and a theory of quantum dimension for VOAs \cite{DJX13}.

\section{Open refinements}

After the Generalized Moonshine and Modular Moonshine conjectures were proposed, some refinements have been proposed based on experimental evidence and theoretical developments.

For Generalized Moonshine, it is natural to ask if the ambiguous constants in claims 2 and 4 can be refined by a good choice of functions $Z(g,h;\tau)$.  Norton's original proposal in \cite{N87} suggested that the constants can be refined to roots of unity, and a later proposal in \cite{N01} asserted that these roots of unity may be taken to be 24th roots of unity.  More recently, the constants were conjectured to be controlled by a nontrivial ``Moonshine Anomaly'' $\alpha^\natural \in H^3(\bM,U(1))$ in \cite{GPV13}, and Norton's assertion amounts to the claim that this anomaly has order 24.  The refined behavior in claims 2 and 4 are not particularly special in a moonshine sense, as it is expected for any holomorphic vertex operator algebra.

There is some progress on this refinement of Generalized Moonshine:
\begin{enumerate}
\item In \cite{C17a}, Theorem 6.2 asserts that the constants relating Hauptmoduln may be taken to be roots of unity.  However, there is still no control over the constant functions.  Also, the techniques used to prove this theorem only work for $V^\natural$ (and other holomorphic vertex operator algebras with the same character, but these were conjectured to not exist in the introduction to \cite{FLM88}).  A similar result for other holomorphic vertex operator algebras would require additional development in the theory of regularity for fixed-points.
\item In \cite{JF17}, the main result is that the Moonshine Anomaly has order 24.  However, at the moment I am writing this, the proof seems to rest on still-conjectural parts of a dictionary relating the theories of conformal nets and vertex operator algebras.
\end{enumerate}

There have been some intriguing developments regarding connections between elliptic cohomology and monstrous moonshine since the 1980s, but there are still essential ingredients missing from the story.  Essentially, there is a conjectural connection through conformal field theory (proposed in \cite{S87}, further developed in \cite{BT99} and ongoing work of Stolz and Teichner), and a connection through higher characters (also mentioned in \cite{S87}, following the theory  developed in \cite{HKR00}).  Based on the introduction of \cite{G07}, it seems that all of the claims of the Generalized Moonshine conjecture except the Hauptmodul claim have had a known basis in elliptic cohomology theory since the 1990s.  In particular, claims 2, 4, and 5 in the Generalized Moonshine conjecture (namely those that don't involve $V(g)$ or the Hauptmodul condition) amount to the claim that the collection of functions $Z(g,h;\tau)$ describe an element of twisted equivariant cohomology of a point, in the sense of \cite{GKV95}.  That is, $Z \in \Ell^\alpha_{\bM}(pt)$.  If we ignore positivity, Claim 1, which connects $Z$ to $V(g)$, amounts to the claim that $Z$ is an element in twisted Devoto equivariant Tate K-theory of a point, i.e., $Z \in K_{Dev,G}^\alpha(pt)$.

There have been more refinements in recent years.  First, Lurie has pointed out in \cite{L09} that elliptic cohomology can be made genuinely 2-equivariant, in the sense that we may consider a more functorial approach to equivariance, and also allow categorical groups.  This would let us shift from an ad hoc definition of $Ell^\alpha_{\bM}(pt)$ to a natural definition of $Ell_{\widetilde{\bM}}(pt)$, where $\widetilde{\bM}$ is a categorical group defined by a 3-cocycle representing $\alpha^\natural$, but unfortunately this new definition has yet to fully appear in public.  Second, Ganter has employed equivariant Hecke operators, first seen as operations on elliptic cohomology in \cite{B90}, to explore the genus zero condition from a topological perspective in \cite{G07}.  She defined an exponential operation
\[ f \mapsto \Lambda^{(2)}_t(f) = \exp(-\sum_{k>0} T_k(f) t^k) \]
taking $Ell^\alpha_G$ to $\bigoplus_{n \geq 0} Ell^{-n\alpha}_G t^n$ by a formal sum of Hecke operators, and interpreted replicability as the equality $p(f(p) - f(q)) = \Lambda^{(2)}_{-p}(f(q))$.  By expanding at a cusp, this formula comes very close to the twisted denominator formulas for Monstrous Lie algebras $\fm_g$ given in \cite{GM1} and \cite{GM2}.  The main difference seems to be the fact that the Weyl denominator of $\fm_g$ is given by $Z(1,g;\sigma) - Z(g,1;\tau)$ rather than $Z(1,g;\sigma) - Z(1,g;\tau)$.  Assuming this can be fixed in a natural way, this condition is essentially the statement that $Z$ is ``self-exponential'' at chromatic level 2, i.e., the exponential operation $\Lambda^{(2)}$ basically takes $Z$ to a slightly altered version of itself.  At chromatic level 1, we have similar behavior with line bundles, since the exterior algebra of a line satisfies $\Lambda^*(L) = 1 + L$.  That is, self-exponential objects are ``exceptionally small'' in a way that is highly nontrivial in the case of elliptic objects.  We then reformulate Generalized Moonshine as:

\begin{conj}
$V^\natural$ defines an element of $\Ell^{\alpha^\natural}(B\bM)$ that is self-exponential.
\end{conj}

The main reason this is still conjectural is because the words ``defines'' and ``self-exponential'' are still rather ill-defined.

For Modular Moonshine, the behavior that is now known to hold for Fricke-invariant classes of prime order is also expected to hold for all Fricke-invariant classes.  That is, we have:

\begin{conj} \label{conj:composite-modular-moonshine} (essentially \cite{BR96} open problems 3 and 4)
For each Fricke element $g \in \bM$, we have $\hat{H}^1(g,V^\natural_{\bZ}) = 0$.  Furthermore, the vertex algebra $\hat{H}^0(g,V^\natural_\bZ)$ over $\bZ/|g|\bZ$ admits an action of $C_\bM(g)$ by automorphisms, such that for any $h \in C_\bM(g)$ satisfying $(|g|,|h|)=1$, the graded Brauer character of $h$ on $\hat{H}^0(g,V^\natural_\bZ)$ is equal to the $q$-expansion of the McKay-Thompson series $T_{gh}(\tau)$.
\end{conj}

The case where $g$ has prime order amounts to the Modular Moonshine conjecture, which is settled.  However, \cite{BR96} Theorem 4.7 also yields the vanishing of $\hat{H}^1(g,V^\natural_{\bZ})$ for $g$ in the composite order classes 15A and 21A.  This conjecture does not seem to be far beyond the reach of current technology.

\section{Unifications}

It is reasonable to ask whether the two conjectures that extend Monstrous Moonshine are really part of some larger, unified Moonshine.  Borcherds proposed one possible unifying conjecture, which we rephrase slightly:

\begin{conj} \label{conj:unified-moonshine} (\cite{B98} Conjecture 7.1)
There is a rule that assigns to each element $g \in \bM$ a $\frac{1}{|g|}\bZ$-graded super-module over $\bZ[e^{2\pi i /|g|}]$, equipped with an action of a central extension $(\bZ/|g|\bZ).C_\bM(g)$, satisfying the following properties:
\begin{enumerate}
\item ${}^1\hat{V}$ is a self-dual integral form of $V^\natural$.
\item If $g$ and $h$ commute, then $\hat{H}^*(\hat{g},{}^h\hat{V}) \cong {}^{gh}\hat{V} \otimes \bZ/|g|\bZ$ for some lift $\hat{g}$ of $g$.
\item If $g$ is Fricke (i.e., $T_g(\tau) = T_g(\frac{-1}{N\tau})$ for some $N \in \bZ_{>0}$), then
\begin{enumerate}
\item ${}^g\hat{V} \otimes \bZ/|g|\bZ \cong \hat{H}^0(g,{}^1\hat{V})$ as vertex algebras over $\bZ/|g|\bZ$ with $C_\bM(g)$-action.  In particular, $\hat{H}^1(g,{}^1\hat{V}) = 0$.
\item ${}^g\hat{V} \otimes \bC \cong V^\natural(g)$ as graded $(\bZ/|g|\bZ).C_\bM(g)$-modules.
\end{enumerate}
\item ${}^g\hat{V}$ is ``often'' a vertex superalgebra with a conformal vector and a self-dual invariant symmetric bilinear form.
\item If $g$ and $h$ commute and have coprime order, then the graded Brauer character of some lift of $g$ on ${}^h\hat{V}$ is equal to $T_{gh}(\tau)$.
\end{enumerate}
\end{conj}

This conjecture is not a true unification, since it does not fully encompass Generalized Moonshine.  In particular, the $SL_2(\bZ)$ and Hauptmodul claims are missing, and there are no claims about non-Fricke twisted modules.  These can be added by hand, but without some additional structure, they do not seem to blend into the rest of the conjecture naturally.

It may be rewarding to consider how much of Borcherds's conjecture we can prove with current technology.

First, we may assign to ${}^1\hat{V}$ a self-dual integral form $V^\natural_\bZ$ with $\bM$-symmetry.  We have existence from Theorem 2.4.2 of \cite{C17}, but do not have a good uniqueness result yet.

Second, we can get pretty close to a candidate object for ${}^g\hat{V}$ when $g$ is Fricke:

\begin{thm}
Let $g \in \bM$ be a Fricke element, and let $N$ be the level of $T_g(\tau)$.  Then there is a $\bZ[\frac{1}{N}, e^{2\pi i/N}]$-form of the irreducible twisted module $V^\natural(g)$, with an action of $V^\natural_\bZ \otimes \bZ[\frac{1}{N}, e^{2\pi i/N}]$, and compatible symmetries by $(\bZ/|g|\bZ).C_\bM(g)$.
\end{thm}
\begin{proof}
By Theorem 1 of \cite{PPV17} (in the anomaly-free case) and Proposition 3.2 of \cite{C17a} (in the anomalous case), the cyclic orbifold dual $V^\natural/g$ is isomorphic to $V^\natural$.  Then by Proposition 1.4.1 of \cite{C17}, there is a self-dual $\bZ[\frac{1}{N}, e^{2\pi i/N}]$-form of the abelian intertwining algebra ${}^g_NV^\natural$ with an action of $(\bZ/N\bZ).C_\bM(g)$ by homogeneous automorphisms.  The central $\bZ/N\bZ$ acts by scalar multiplication, and the description of cyclic orbifolds in \cite{vEMS} allows us to reduce this to $\bZ/|g|\bZ$.  This abelian intertwining algebra contains the $\bZ[\frac{1}{N}, e^{2\pi i/N}]$-form of $V^\natural(g)$ as a sub-$(V^\natural)^g_\bZ \otimes \bZ[\frac{1}{N}, e^{2\pi i/N}]$-module.
\end{proof}

This theorem yields a candidate for ${}^g\hat{V} \otimes_{\bZ[e^{2\pi i/|g|}]} \bZ[\frac{1}{N}, e^{2\pi i/N}]$.  To obtain ${}^g\hat{V}$ itself, we would need to choose a suitable stable lattice whose reduction to $\bZ/|g|\bZ$ is equivariantly isomorphic to $\hat{H}^0(g,{}^1\hat{V})$.  Even assuming Conjecture \ref{conj:composite-modular-moonshine}, we would have to glue characteristic zero representations to characteristic $p^k$ representations for various primes $p$.  As Jeremy Rickard pointed out to me at \cite{R18}, the fact that the characteristic zero characters are equal to the Brauer characters is far from sufficient for the existence of a compatible mixed characteristic representation.

The conjectured vertex superalgebra structure on ${}^g\hat{V}$ is a bit of a mystery.  Borcherds claims in \cite{B98} that we get such a structure for the classes 1A, 2B, 3B, 3C, 5B, 7B, and 13B, but apparently not 2A and 3A.  Although I have not seen Borcherds's constructions, it is natural to extend the pattern of prime-order non-Fricke classes:

\begin{conj}
If $g$ is non-Fricke, then ${}^g\hat{V}$ is a vertex superalgebra.
\end{conj}

One feature that distinguishes the non-Fricke elements of $\bM$ is the fact that the cyclic orbifold dual of $V^\natural$ along a non-Fricke element is always the Leech lattice vertex operator algebra, as conjectured in \cite{T93} and proved in Theorem 4.5 of \cite{C17a}.  It is therefore reasonable to expect the following:

\begin{conj}
When $g$ is non-Fricke, there is a natural construction of the vertex superalgebra structure on ${}^g\hat{V}$ using the Leech Lattice.
\end{conj}

Unfortunately, this is rather vague as stated.  If we had a more precise formulation, we might be well on our way toward a construction.  

\section{Speculation about gluing}

In the previous section, we noted that we don't have an obvious way to glue a $\bZ[\frac{1}{N}, e^{2\pi i/N}]$-form of $V^\natural(g)$ to the Tate cohomology of $V^\natural_\bZ$ to form ${}^g\hat{V}$.  This problem could be resolved if we had a canonical way to lift Tate cohomology to mixed characteristic.  Tate cohomology does not in general admit such canonical lifts, but we could hope that there is an additional structure on $V^\natural_\bZ$, like a non-linear descent datum, that produces such a lift.  I have been unable to find a structure of this precise form in the mathematical literature, but it bears some resemblance to various proposals for algebraic structures over the mythical ``field of one element''.  While $\bF_1$-structures do not have a universally recognized interpretations, one feature many proposals share is that they let one treat mixed-characteristic and equal-characteristic objects on a roughly equal footing.

\begin{conj} \label{conj:F1}
There is an $\bF_1$-module structure on $V^\natural_\bZ$ with $\bM$-symmetry, where by ``$\bF_1$-module structure'' we mean that for each $g \in \bM$, there is a distinguished lift of the $(\bZ/|g|\bZ).C_\bM(g)$-action on $\hat{H}^*(g,V^\natural_\bZ)$ (as a graded module) to the big Witt vectors $W_\infty \hat{H}^*(g,V^\natural_\bZ)$.
\end{conj}

In \cite{B08}, $\Lambda$-rings are proposed as an interpretation for the term ``$\bF_1$-algebra''.  It is then reasonably natural to presume modules ought to have a correspondingly non-linear interaction with Witt vectors.  Conjecture \ref{conj:F1} would yield candidates for the proposed modules ${}^g\hat{V}$, but it does not naturally lead to any exceptional modularity properties, such as the Hauptmodul property for $p$-singular elements.

\section{Connections to newer moonshines}

In 2014, an amateur mathematician, Tito Piezas III, posted a question proposing a Thompson moonshine phenomenon on \texttt{mathoverflow.net}, a question and answer site for mathematicians.  The initial observation was that the weight 1/2 weakly holomorphic modular form
\[ b(\tau) = q^{-3} + 4 - 240q + 26760q^4 - 85995q^5 + \cdots \]
has coefficients that easily decompose into combinations of dimensions of irreducible representations of the Thompson sporadic group $Th$.  It was quickly pointed out that there is already a Thompson moonshine, coming from Generalized Moonshine for the class 3C, where the twisted module has character $\sqrt[3]{j(\tau/3)} = q^{-1/9} + 248q^{2 /9}+ 4124q^{5/9} + \cdots$ and an action of $\bZ/3\bZ \times Th$.  These two moonshines are connected by the Harvey-Moore-Borcherds correspondence, where up to some adjustment with eta functions and reparametrization, the Borcherds product attached to $b(\tau)$ is equal to the character of the irreducible 3C-twisted $V^\natural$-module.  Generalized Moonshine then gives an explanation for the appearance of Thompson representations attached to powers of $q$ in $b(\tau)$ that are squares.  However, there is at this time no explanation for the representation-theoretic content of other coefficients.

Piezas's observation was refined to a concrete conjecture in \cite{HR15}, and the conjecture was solved in \cite{GM16}.  The function $b(\tau)$ was replaced with
\[ 2b(\tau) + 240\theta(\tau) = 2q^{-3} + 248 + 2 \cdot 27000q^4 - 2\cdot 85995q^5 + \cdots \]
to remove sign ambiguities, and the theorem is that there is a graded $Th$-super-module $V = \bigoplus_{n \geq -3} V_n$ whose graded dimension is $2b(\tau) + 240\theta(\tau)$, such that for any $h \in Th$, the trace function $f_h := \sum_{n \geq -3} \Tr(h|V_n) q^n$ is a modular form of weight 1/2 for some congruence group controlled by the order of $h$.  We find that the Borcherds product attached to $2b(\tau) + 240\theta(\tau)$ is equal to $\eta^{248}Z(1,g;\tau)^2$ for $g$ an element of class 3C in $\bM$.  More generally, one may ask whether Borcherds products take $f_h$ to suitably altered Generalized Moonshine functions, or perhaps Modular Moonshine functions, and there are some promising partial answers.  However, the nature of this correspondence is still a complete mystery.

The presence of a square in $\eta^{248}Z(g,1;\tau)^2$ is quite similar to the situation in Mathieu Moonshine, where the Borcherds lift of the elliptic genus of a K3 surface yields the Igusa cusp form $\Phi_{10}$.  In this case, $1/\Phi_{10}$ is the square of a function with representation-theoretic interest, namely the Weyl-Kac-Borcherds denominator of a hyperbolic-type Lie algebra \cite{GN95}.  I do not know a good explanation for the natural appearance of squares of characters of infinite dimensional algebras.

Another mystery comes from a conjecture of Borcherds that was resolved in \cite{GL13}, asserting that the Modular Moonshine vertex algebra for 3C is isomorphic to the $E_8$-vertex algebra over $\bF_3$.  This means the vertex algebra has symmetries of $E_8(\bF_3)$, which is far larger than $Th$.  While $Th$ embeds in $E_8(\bF_3)$, it does not embed in $E_8(\bC)$.  This seems to be evidence against the idea that the conjectural module ${}^g\hat{V}$ can be made by simply taking an integral form of the $E_8$ vertex algebra.

Finally, a generalization of Thompson moonshine to other weight 1/2 forms, called ``skew-holomorphic moonshine'', was announced by Duncan, Harvey, and Rayhoun in 2016.  While we are still waiting for a paper, it appears that applying the Borcherds product construction to the certain vector-valued forms yields modular functions that can range over all Fricke-invariant McKay-Thompson series.  This suggests the existence of a connection to modular moonshine for Fricke $g$.  Once again, we don't seem to have any theoretical underpinning of the underlying objects.

\end{document}